\documentclass[a4paper,10pt]{article}
\usepackage{amsmath,amssymb,amsfonts}
\usepackage{epsfig,color}
\usepackage{mathrsfs}
\usepackage{hyperref}
\usepackage[utf8]{inputenc}
\usepackage{dsfont}
\usepackage[english]{babel}
\usepackage{amsthm}
\usepackage{fancyhdr}

\usepackage[a4paper,top= 4cm,left=3.5cm,right=3.5cm,
bindingoffset=5mm]{geometry}

\usepackage[usenames,dvipsnames]{xcolor}

\theoremstyle{definition}
\newtheorem{defn}{Definition}[section]

\theoremstyle{plain}
\newtheorem{thm}{Theorem}[section]

\theoremstyle{remark}
\newtheorem{rmk}{Remark}[section]

\newcommand{\X}{\mathbf{X}}
\newcommand{\Y}{\mathbf{Y}}
\newcommand{\NN}{\mathbf{N}}
\newcommand{\E}{\mathbb{E}}
\newcommand{\SSw}{\mathbf{S}}

\def\P{{\rm I\kern-0.16em P}}

\def    \bs     {\boldsymbol}

\title{Multidimensional limit theorems for homogeneous sums: a general transfer principle}

\begin{document}

\title{Multidimensional limit theorems for homogeneous sums: a general transfer principle}
\author{I. Nourdin, G. Peccati, G. Poly, R. Simone}
%
\maketitle
\date{}

\begin{abstract} 
The aim of the present paper is to establish the multidimensional counterpart of the \textit{fourth moment criterion} for homogeneous sums in independent leptokurtic and mesokurtic random variables (that is, having  positive and zero fourth cumulant, respectively), recently established in \cite{NPPS} in both the classical and in the free setting. As a consequence, the transfer principle for the Central limit Theorem between Wiener and Wigner chaos can be extended to a multidimensional transfer principle between vectors of homogeneous sums in independent commutative random variables with zero third moment and with non-negative fourth cumulant, and homogeneous sums in freely independent non-commutative random variables with non-negative fourth cumulant. 
\end{abstract}

\textbf{Subject classification: } 60F17, 60F05, 46L54 \\
\textbf{Keywords: } Fourth Moment Phenomenon; Free Probability; Homogeneous Sums; Multidimensional Limit Theorems; Wiener Chaos; Wigner Chaos


\section{Introduction}
The \textit{fourth moment phenomenon} is a collection of probabilistic results, allowing one to deduce central limit theorems (in both the classical and free probability settings) for a sequence $\{X_n : n\geq 1\}$ of non-linear functionals of a random field, merely by controlling the sequences $\{\E X^2_n = n\geq 1\}$ and $\{\E X_n^4 : n\geq 1\}$ of the first two even moments. First discovered in \cite{NualartPeccati} in the context of non-linear transformations of Gaussian fields, such a phenomenon is gaining an increasing interest in the mathematical community, due to its wide range of applications. The reader is referred to the monograph \cite{NourdinPeccatilibro} for an introduction to the topic. See \cite{chen, chenpoly, pecsur} for recent surveys, as well as \cite{web} for a constantly updated account of the mathematical literature on this topic.\\

The present paper focuses on several multidimensional consequences of the fourth moment phenomenon, both in the classical and in the free setting. In particular, our goal is to apply the results from \cite{NPPS} in order to generalize the transfer principle for the Central Limit Theorem between Wiener and Wigner chaoses, established in \cite{NouSpeiPec}. Transfer principles of this type can be potentially very useful for establishing free counterparts to well-understood results in the classical settings . For instance, a remarkable example is given by the free version of the {\it Breuer-Major Theorem} pointed out in \cite{KempNourdinPeccatiSpeicher}.

The main results of the present note complete the findings of \cite{NPPS}, where the authors have described a new large class of random variables to which the \textit{fourth moment phenomenon} applies. This class contains in particular homogeneous sums in independent copies of a leptokurtic or mesokurtic random variable $X$ (that is, a random variable with positive or zero fourth cumulant respectively, in both the probability settings).  The proofs of these generalized \textit{Fourth Moment Theorems} involve some combinatorial arguments, since they rely on  new formulae for the fourth moment of homogeneous sums. In order to work out the proof,  an additional assumption  in the classical case has been needed, resulting in the requirement $\E[X^3]=0$. The same assumptions have allowed the authors to establish that homogeneous sums in leptokurtic or mesokurtic random variables verify also an invariance principle for central convergence, customarily referred to as \textit{universality phenomenon}, in both the probability settings. \\

Further results dealing with universality and Fourth Moment Theorem in the free  setting include \cite{Solesne2,Solesne,NourdinDeya, NourdinPeccati1, Simone}, and for the commutative framework \cite{Guillaume1, NourdinPeccatiReinert, NourdinPeccatiReveillac, PeccatiZheng2, PeccatiZheng1}. Further, the analysis of the fourth moment phenomenon for infinitely divisible laws has been addressed in \cite{Arizmendi1} while, more recently, limit theorems encompassing the fourth moment and the universality phenomena have been investigated also in the setting of the random graphs colouring problem \cite{Diaconis}. \\

As already mentioned, the proofs we present do not require any additional techniques with respect to those developed in \cite{NPPS}, of which this paper is meant to be a sequel. The main results that will be proved in the present work are Theorem \ref{ComponentJoint} and its free version Theorem \ref{ComponentJointFree}, in which we prove that joint and componentwise central convergence are equivalent for vectors of homogeneous sums in independent copies of leptokurtic or mesokurtic variables, both in the commutative and in the non-commutative framework. The combination of these results lead to the formulation of the general transfer principle for (vectors of) such random variables, achieved via Theorem \ref{Transfer}.

\section{Preliminaries}

Before stating our main results, some preliminary notations and definitions need to be fixed. For any unexplained concept or result pertaining to free probability theory, the reader is referred to the fundamental references \cite{Speicher,Voiculescu}.

For every $n\in \mathbb{N}$, set $[n] := \{1,\dots,n\}$. Let $X$ be a random variable defined on a fixed probability space $(\Omega,\mathcal{F},\mathbb{P})$. Unless otherwise specified, it will be always assumed that $X$ satisfies the following assumptions, that will be referred to as Assumption {\bf (1)}:
\begin{itemize}
\item[(i)] $X$ is centered and has unit variance;
\item[(ii)] $\E[X^3]=0$;
\item[(iii)] there exists $\epsilon > 0$ such that $\E[\vert X\vert^{4+\epsilon}] < \infty$.
\end{itemize}

Given  a sequence $\X=\{X_i\}_{i\geq 1}$  of independent copies of $X$ (i.i.d. for short), 
we will consider random variables having the form of multilinear homogeneous  polynomials of degree $d\geq 2$:
\begin{eqnarray}\label{e:alv}
Q_{\X}(f) &=&\sum_{i_1,\dots,i_d =1}^n f(i_1,\dots,i_d) X_{i_1}\cdots X_{i_d},\label{F}
\end{eqnarray}
where the mapping $f:[n]^d \rightarrow \mathbb{R}$ is an {\it admissible kernel}, in the sense of the following definition.

\begin{defn}
\label{Admissible}
For a given degree $d\geq 2$ and some integer $n\geq 1$, a function $f:[n]^d\to\mathbb{R}$ is said to be an {\it admissible kernel} if the following properties are satisfied:
\begin{itemize}
\item[(i)] $f$ vanishes on diagonals, that is, $f(i_1,\dots,i_d)=0$ whenever  $i_j=i_k$ for some $k\neq j$;
\item[(ii)] $f$ is symmetric, namely $f(i_1,\dots,i_d)=f(i_{\sigma(1)},\dots,i_{\sigma(d)})$ for any permutation $\sigma$ of $\{1,\dots,d\}$ and any $(i_1,\dots,i_d)\in [n]^d$;
\item[(iii)] $f$ satisfies the normalization: $$d!\sum\limits_{i_1,\dots,i_d =1}^n f(i_1,\dots,i_d)^2=1.$$
\end{itemize}
\end{defn}

Since $f$ is an admissible kernel and $X$ satisfies Assumption {\bf (1)},  the homogenous sum $Q_{\X}(f)$ verifies $\E[Q_{\X}(f) ]=0$ and $\E[Q_{\X}(f)^2]=1$.

\begin{rmk} Note that the symmetry and the normalization assumptions on $f$ are introduced for mere convenience: indeed, given a function $f:[n]^d\to\mathbb{R}$ that is vanishing on diagonals, it is always possible to generate an admissible kernel $\tilde{f}$ by first symmetrizing $f$ and then by properly renormalizing it. 
\end{rmk}

As already discussed in the Introduction, the goal of the present paper is to complete the findings of \cite{NPPS}, in particular providing an extension of the results therein to the multidimensional setting. In order to achieve our goals, we shall need the following definitions.

\begin{defn}\label{defcom}
Let $X$ be a random variable verifying Assumption {\bf (1)}, $\X=\{X_i\}_{i\geq 1}$ be a sequence of independent copies of $X$.
\begin{itemize}
\item[(a)] We say that $X$ satisfies \textit{the Fourth Moment Theorem at the order $d\geq 2$} (for normal approximations of homogeneous sums) if, for every sequence $f_n:[n]^d\to\mathbb{R}$ of admissible kernels, the following statements are equivalent for $n\to\infty$:
\begin{enumerate}
\item[(i)] $Q_{\X}(f_n) \xrightarrow{\text{\rm Law}} \mathcal{N}(0,1)$. 
\item[(ii)]  $\E[Q_{\X}(f_n)^4]\to \E[N^4]=3$, where $N \sim \mathcal{N}(0,1)$.
\end{enumerate}
\item[(b)] $X$ is said to be \textit{universal at the order $d$} (for normal approximations of homogeneous sums) if, for any sequence $f_n:[n]^d\to\mathbb{R}$ of admissible kernels, $Q_{\X}(f_n) \xrightarrow{\text{\rm Law}}\mathscr{N}(0,1)$ implies, as $n\to\infty$,
$$ \tau_n:= \max_{ i=1,\dots, n} \mathrm{Inf}_i(f_n) \longrightarrow 0, $$
where $\mathrm{Inf}_i(f_n):=\sum\limits_{i_2,\ldots,i_d=1}^n f_n(i ,i_2,\ldots,i_d)^2$ is the $i$-th \textit{influence function} of $f_n$.\\
\end{itemize}
\end{defn}

As shown originally in \cite{NualartPeccati} and \cite{NourdinPeccatiReinert}, if $X$ is normally distributed then it verifies both Points (a) and (b) in Definition \ref{defcom}. In this case, the corresponding homogeneous sums $Q_{\bf X}(f_n)$ are said to be elements of the $d$th (Gaussian) {\it Wiener chaos} associated with ${\bf X}$ (see e.g. \cite{PecTaq} for an introduction to these concepts). We observe that in \cite{NourdinPeccatiReinert} it is proved that the Gaussian Wiener chaos is also universal with respect to Gamma approximations. A crucial role in the proof of the universal behavior of the Gaussian Wiener chaos  has been played by the findings in \cite{Mossel}, where the authors have measured the proximity in law between homogeneous sums in terms of \textit{influence functions}. The next statement records the estimates from \cite{Mossel} that are needed for our discussion.

\begin{thm}\label{invMossel}
Let $\bs{X}=\{X_i\}_{i\geq 1}$ and $\mathbf{Y}=\{Y_i\}_{i\geq 1}$  be sequences of independent centered random variables on a fixed probability space, with unit variance and uniformly bounded moments of every order. Then, for $d\geq 1$ and for every sequence of admissible kernels $f_n:[n]^d \rightarrow \mathbb{R}$, 
\begin{equation*}
\mathbb{E}[Q_{\bs{X}}(f_n)^m] - \mathbb{E}[Q_{\bs{Y}}(f_n)^m] = \mathcal{O}\big( \sqrt{\tau_n}\big)\quad \forall \, m \in \mathbb{N},\\
\end{equation*}
where $\tau_n := \max\limits_{i=1,\dots,n}\mathrm{Inf}_{i}(f_n)$.
\end{thm}

\begin{rmk}
Note that the issue of universality is relevant only for homogeneous sums of degree $d\geq 2$, since no invariance principle holds for homogeneous sums of degree $d=1$ (see \cite{NourdinPeccatiReinert}). For degrees $d\geq 2$ and in view of Theorem \ref{invMossel}, one could alternatively define $X$ to be universal at the order $d$ if, for any sequence $f_n:[n]^d\to\mathbb{R}$ of admissible kernels, $Q_{\X}(f_n) \xrightarrow{\text{\rm Law}}\mathscr{N}(0,1)$ implies $Q_{\mathbf{Z}}(f_n) \xrightarrow{\text{\rm Law}}\mathscr{N}(0,1)$
for every sequence $\mathbf{Z}$ of independent copies of a centered random variable having unit variance.\\
\end{rmk}

We now turn to the non-commutative setting. Consider a fixed non commutative probability space $(\mathcal{A},\varphi)$, where $\mathcal{A}$ is a unital $\ast$-algebra, and $\varphi$ is a unital, faithful and positive trace. Let $Y$ be a random variable on it that is centered and that has unit variance, that is, $\varphi(Y)=0$ and $\varphi(Y^2)=1$. In this case, it will be said for short that $Y$ satisfies Assumption {\bf (2)}. \\

If $\Y=\{Y_i\}_{i\geq 1}$ is a sequence of freely independent copies of $Y$,   the free counterpart to random variables of the form \eqref{F} are self-adjoint elements of the type:
\begin{eqnarray}
Q_{\Y}(f) &=&\sum_{i_1,\dots,i_d =1}^n f(i_1,\dots,i_d) Y_{i_1}\cdots Y_{i_d},\label{Fnoncom}
\end{eqnarray}
where $f$ is an admissible kernel.

From Assumption {\bf (2)} and the properties of $f$, it follows that $\varphi(Q_{\Y}(f) )=0$ and $d!\varphi(Q_{\Y}(f)^2)=1$.\\

\begin{rmk}
In the free setting, the natural choice for the coefficient of a homogeneous sum would be a \textit{mirror symmetric function}, namely a kernel $f:[n]^d \rightarrow \mathbb{C}$ such that $f(i_1,i_2, \dots, i_d) = \overline{f(i_d,\dots, i_2,i_1)}$ for every $i_1,\dots, i_d \in [n]$, with $\bar{z}$ denoting the complex conjugate of $z$. This assumption is the weakest possible to ensure that the element $Q_{\Y}(f)$ is self-adjoint. However, the forthcoming discussion will heavily rely on the universality property of the Wigner Semicircle law that has been so far established only for homogeneous sums with symmetric real-valued coefficients: indeed, both in \cite{Solesne2} and in \cite{NourdinDeya}, counterexamples to the universality for mirror symmetric kernels have been provided.  \\
\end{rmk}

For several reasons, the semicircular distribution is considered  as the non-commutative analogue of the Gaussian distribution: for instance, it is the limit law for the free version of the Central Limit Theorem, and joint moments of a semicircular system satisfy a Wick-type formula \cite{Speicher}. A very interesting fact for our purposes is that the semicircular law satisfies both the Fourth Moment Theorem and the universality property as to semicircular approximations (see \cite{NourdinDeya}), thus justifying the following definitions.  

\begin{defn}\label{defnoncom}
Let $Y$ satisfy Assumption {\bf (2)} and let $S$ be a standard semicircular random variable, for short $S \sim \mathcal{S}(0,1)$. Let $\bs{Y}=\{Y_i\}_{i\geq 1}$ and $\bs{S}=\{S_i\}_{i\geq 1}$ be sequences of  freely independent copies of $Y$ and $S$ respectively. For a fixed $d\geq 2$ and for every $n\geq 1$, let $f_n:[n]^d\to\mathbb{R}$ be an admissible kernel.  
\begin{itemize}
\item[(a)] We say that $Y$ satisfies the {\it free Fourth Moment Theorem} of order $d$ (for central approximations) if, for any sequence $f_n:[n]^d\to\mathbb{R}$ of admissible kernels, the following statements are equivalent as $n\to\infty$:
\begin{itemize}
\item[(i)] $d!^2\varphi(Q_{\Y}(f_n)^4)\rightarrow \varphi(S^4) = 2, \quad S \sim \mathcal{S}(0,1)$;
\item[(ii)] $ \sqrt{d!} \,Q_{\Y}(f_n) \xrightarrow{\text{\rm Law}}\mathcal{S}(0,1).$
\end{itemize}
\item[(b)]  We say that $X$ is {\it free universal} at the order $d$ (for central approximations) if, for any sequence $f_n:[n]^d\to\mathbb{R}$ of admissible kernels,  $\sqrt{d!} \,Q_{\Y}(f_n) \xrightarrow{\text{\rm Law}}\mathcal{S}(0,1)$ implies, as $n\rightarrow \infty$,
$$\tau_n = \max_{i=1,\dots,n}\mathrm{Inf}_i(f_n) \rightarrow 0.$$
\end{itemize}
\end{defn}

The findings established with \cite[Theorem 3.2]{Simone} provide a general multidimensional version of Theorem \ref{invMossel} in the free probability setting. Here the invariance principle will be formulated via Theorem \ref{Multiinvariance2} for estimating the proximity in law between vectors of homogeneous sums.

\begin{thm}
\label{Multiinvariance2} 
Let $\bs{X} = \{X_i\}_{i \geq 1}$  and $\bs{Y}=\{Y_j\}_{j\geq 1}$ be sequences of freely independent random variables, centered and with unit variance, freely independent between each other. Assume further that $\bs{X}$ and $\bs{Y}$ are composed of random variables with uniformly bounded moments, that is, for every integer $r\geq 1$,
$$ \sup_{i\geq 1} \varphi(|X_i|^{r}) < \infty \qquad (\text{resp. } \sup_{i\geq 1} \varphi(|Y_i|^{r}) < \infty).$$
For every integer $k\geq 1$, for every choice of $\bs{m}_s= (m_{s,1},\dots,m_{s,p}) \in \mathbb{N}^p$ for $s=1,\dots,k$, if $ \bs{Q}_n(\bs{Y}) = (Q_{\bs{Y}}(f_n^{(1)}),\dots,Q_{\bs{Y}}(f_n^{(p)}))$ denotes a vector of homogeneous sums with  admissible kernel  $f_n^{(j)}: [n]^{d}\rightarrow \mathbb{R}$ for every $j=1,\dots,p$, then:
\begin{small}
\begin{align}
\varphi\big(\bs{Q}_n(\bs{X})^{\bs{m}_1}\bs{Q}_n(\bs{X})^{\bs{m}_2}\cdots \bs{Q}_n(\bs{X})^{\bs{m}_k} \big)  &- \varphi\big(\bs{Q}_n(\bs{Y})^{\bs{m}_1}\bs{Q}_n(\bs{Y})^{\bs{m}_2}\cdots \bs{Q}_n(\bs{Y})^{\bs{m}_k}\big) \nonumber\\
&= \mathcal{O}\big(\max\limits_{j=1,\dots,p}(\tau_{n}^{(j)})^{\frac{1}{2}}\big),
\end{align}
\end{small}
where $\tau_n^{(j)} = \max\limits_{i=1,\dots,n}\mathrm{Inf}_i(f_n^{(j)})$ and where for $\bs{m}=(m_1,\dots,m_p)\in \mathbb{N}^p$ we have used the standard multi-index notation $\bs{Q}_n(\bs{Y})^{\bs{m}}:= Q_{\bs{Y}}(f_n^{(1)})^{m_1}Q_{\bs{Y}}(f_n^{(2)})^{m_2}\cdots Q_{\bs{Y}}(f_n^{(p)})^{m_p}$.
\end{thm}

\begin{rmk}
It is worth to remark that Theorem \ref{Multiinvariance2} has been originally formulated for more general objects than homogeneous sums, and that in the above simplified formulation for homogeneous sums, it encompasses \cite[Theorem 1.3]{NourdinDeya} corresponding to $p=1$.
\end{rmk}

By virtue of Theorem \ref{Multiinvariance2}, the definition of free universal law can be equivalently reformulated by saying that $X$ is freely universal (at the order $d$) if, for any sequence $f_n:[n]^d\to\mathbb{R}$ of admissible kernels, $\sqrt{d!} \,Q_{\Y}(f_n) \xrightarrow{\text{\rm Law}}\mathcal{S}(0,1)$ implies $\sqrt{d!} \,Q_{\mathbf{Z}}(f_n) \xrightarrow{\text{\rm Law}}\mathcal{S}(0,1)$ for every sequence $\mathbf{Z} = \{Z_i\}_{i\geq 1}$ of freely independent and identically distributed random variables verifying Assumption {\bf (2)}.\\

Theorems \ref{superTeo1} and \ref{superTeo2} below provide the initial impetus for our investtigations; their proofs can be found in \cite{NPPS}. To fix the notation, $\chi_4(X)=\E[X^4]-3$ and $\kappa_4(Y)=\varphi(Y^4)-2$ will denote, respectively, the fourth cumulant of a random variable $X$ satisfying Assumption {\bf (1)} and the fourth free cumulant of a non-commutative random variable $Y$ verifying Assumption {\bf (2)}.

\begin{thm}
\label{superTeo1}
Fix $d\geq 2$ and let $X$ be a random variable satisfying Assumption {\bf (1)}. If $\E[X^4]\geq 3$ (or, equivalently, $\chi_4(X)\geq 0$), then $X$ satisfies the Fourth Moment Theorem and its law is universal at the order $d$ for normal approximations of homogeneous sums, in the sense of Definition \ref{defcom}.
\end{thm}

\begin{thm}
\label{superTeo2}
Fix $d\geq 2$ and consider a random variable $Y$ verifying Assumption {\bf (2)} and such that $\varphi(Y^4) \ge 2$ (or, equivalenty, $\kappa_4(Y) \geq 0$). Then, $Y$ satisfies the free Fourth Moment Theorem and it is free universal at the order $d$ for semicircular approximations of free homogeneous sums. 
\end{thm}

\begin{rmk}[Gamma and Free Poisson approximations]
In \cite{NPPS}, it was shown that, when $d$ is an even integer, any random variable satisfying {\bf (1)} and $\chi_4(X) \geq 0$, is universal and satisfies the Fourth Moment Theorem with respect to the Gamma approximation as well (see also \cite{NourdinPeccati2}). Similarly, in the free setting, every non-commutative random variable satisfying Assumption {\bf (2)} and with $\kappa_4(X) \geq 0$ is both universal and satisfies the Fourth Moment Theorem with respect to the free Poisson approximation, that can be considered as the free counterpart to the Gamma law (see also \cite{NourdinPeccati1}). Unfortunately, so far there is no result proving the equivalence between componentwise and joint convergence for Poisson limits, and hence our strategy cannot deal with Poisson approximations.\\
\end{rmk}

\subsection{Some examples}
As to the classical probability setting, Theorem \ref{superTeo1} supplies us with several examples of laws that satisfy the Fourth Moment Theorem and are universal for central convergence. Here are some examples.
\begin{enumerate}
\item Let $X_1, X_2$ be independent random variables satisfying Assumption {\bf (1)} and such that $\chi_4(X_1), \chi_4(X_2) \geq 0$. Then $Z = X_1 + X_2$ satisfies in turn Assumption {\bf (1)} and $\chi_4(Z) \geq 0$ (due to the additivity property of cumulants), and hence satisfies the Fourth Moment Theorem. As to multiplicative convolution, $W:= X_1 X_2$ satisfies Assumption {\bf (1)} as well. By virtue of the moment-cumulant formula, $\chi_4(W) =  \E[X_1^4]\E[X_2^4] - 3$ and hence, according to Theorem \ref{superTeo1}, for $W$ to satisfy the Fourth Moment Theorem it is sufficient that that at least one of the $X_i$'s satisfies $\chi_4(X_i)\geq 0$.  
\item Every random variable $X$, centered and with unit variance, whose law is infinitely divisible with respect to additive convolution, satisfies $\chi_4(X) = \E[X^4] - 3\geq 0$ (see for instance \cite[Proposition A1]{Arizmendi1bis}). Hence, if $X$ is infinitely divisible and satisfies Assumption {\bf (1)}, the Fourth Moment Theorem for homogeneous sums $Q_{\X}(f_n)$ holds at any order $d\geq 2$. The same necessary condition on the kurtosis has been exploited to study the non-classical infinite divisibility of power semicircular distributions in \cite{Arizmendi2}.

\item For $k\geq 1$, let $H_k(x)$ denote the $k$-th Hermite polynomial and let $N \sim \mathcal{N}(0,1)$. Then,
$$ \E[H_k(N)^4] = |\mathcal{P}_2^{\star}(k^{\otimes 4}) | \geq 3 \; $$
where $ \mathcal{P}_2^{\star}(k^{\otimes 4})$ denotes the set of pairing partitions $\sigma$ of $[4k]$ such  that  every block of $\sigma$ intersect each block of 
$$ k^{\otimes 4} := \{\{1,\dots,k\}, \{k+1,\dots,2k\},\dots,\{2k+1,\dots,3k\},\{3k+1,\dots,4k\}\}$$
in at most one element.  Since $\E[H_k(N)^3]=0$ if $k$ is odd,  for $X=H_k(N)$, Theorem \ref{superTeo1} applies: the techniques so far established do not allow us to infer that  the assumption on the third moment can be dropped.
\end{enumerate}
As to the non-commutative setting, Theorem \ref{superTeo2} establishes that the Fourth Moment Theorem (along with the universality phenomenon) applies, for instance, in the following cases:
\begin{enumerate}
\item
Every random variable $Y$ satisfying Assumption {\bf (2)} and whose law is infinitely divisible with respect to the additive free convolution, satisfies $\kappa_4(Y) = \varphi(Y^4) - 2\geq 0$ (see for instance \cite[Proposition A1]{Arizmendi1bis}). Hence, every  freely infinitely divisible law satisfies the Fourth Moment Theorem (and the universality) as to semicircular and free Poisson approximations of homogeneous sums, at any order $d\geq 2$. 

\item For $k\geq 1$, if $U_k(x)$ denotes the $k$-th Chebyshev polynomial (of the second kind) and $S \sim \mathcal{S}(0,1)$, then:
$$ \varphi[U_k(S)^4] = |\mathcal{NC}_2^{\star}(k^{\otimes 4}) | \geq 2 \; ,$$
where $\mathcal{NC}_2^{\star}(k^{\otimes 4})$ denotes the set of the non-crossing pairings $\sigma$ of the set $[4k]$, such that in each block of $\sigma$, there is at most one element of every block of the interval partition $k^{\otimes 4}$. Note that the universality of the law of $U_k(S)$ for semicircular (and free Poisson) approximations of homogeneous sums can be also established via the approach developed in \cite{Simone}.

\item Let $\mathcal{T}$ be a Tetilla-distributed random variable, namely $\mathcal{T} \stackrel{ \text { Law} }{=} \frac{1}{\sqrt{2}}(S_1 S_2 + S_2 S_1)$, where $S_1, S_2$ are freely independent standard semicircular random variables. Since $\kappa_4(\mathcal{T}) = \frac{1}{2} > 0$,  $\mathcal{T}$ satisfies both the Fourth Moment Theorem and the universality property for semicircular and free Poisson approximations of homogeneous sums, at any order $d\geq 2$ (see \cite{NourdinDeya2}).

\item Let $X \sim \mathcal{G}_q(0,1)$, with $\mathcal{G}_q(0,1)$ denoting the $q$-Gaussian distribution -- that we assume to be defined on an adequate non-commutative probability space $(\mathscr{A}_q, \varphi_q)$. Then, $\kappa_4(X)  = \varphi_q(X^4) - 2 = q$, and hence, if $q \in [0,1]$, $X$ satisfies the Fourth Moment Theorem and the law $\mathcal{G}_q(0,1)$ is universal at any order $d \geq 2$ for semicircular and free Poisson approximations of free homogeneous sums. See \cite[Theorem 3.1 and Proposition 3.2]{qbrownian} for the general Fourth Moment Theorem for integrals with respect to a $q$-Brownian motion of symmetric kernels, for non-negative values of $q$. Equivalently, the fourth moment and the universality phenomena for $X$ can be alternatively deduced from the fact that, for positive values of $q$, the $q$-Gaussian distribution is also freely infinitely divisible \cite{Lehner}.
\end{enumerate}

\section{Main results}

\subsection{Multidimensional Central Limit Theorem in the classical setting}

The next statement corresponds to \cite[Theorem 7.1]{NourdinPeccatiReinert}, where the authors have provided an explicit error bound for the distance in law between a vector of the type $$(Q_{\bs{X}}(f_n^{(1)}),\dots,Q_{\bs{X}}(f_n^{(m)}))$$ and its Wiener-chaos counterpart $(Q_{\bs{N}}(f_n^{(1)}),\dots,Q_{\bs{N}}(f_n^{(m)}))$. For the sake of completeness, it is worth to underline that a first attempt of extending Theorem \ref{invMossel} to the multidimensional case  has been carried out in \cite[Theorem 4.1]{Mossel2}, in the case one of the sequences lives on a discrete probability space.

\begin{thm}\label{MultiMossel}
Let $m,d\geq 1$. Let $\bs{X}=\{X_i\}_{i\geq 1}$ be a sequence of centered independent random variables, with unit variance, whose third moments are uniformly bounded (namely, such that there exists $\beta >0$ such that $\sup\limits_{i\geq 1}\E[|X_i|^{3}] < \beta$). For $j=1,\dots,m$, let $f_n^{(j)}:[n]^d \rightarrow \mathbb{R}$ be an admissible kernel according to Definition \ref{Admissible}. If $\bs{N}= \{N_i\}_{i\geq 1}$ denotes a sequence of i.i.d. standard Gaussian random variables,  for every thrice differentiable function $\psi:\mathbb{R}^m \rightarrow \mathbb{R}$, with $\|\psi^{\prime\prime \prime} \|_{\infty} < \infty$, there exists a constant $C=C(\beta, m, d, \psi)$ such that:
$$ \big| \E[\psi(Q_{\bs{X}}(f_n^{(1)}),\dots,Q_{\bs{X}}(f_n^{(m)}) )] -   \E[\psi(Q_{\bs{N}}(f_n^{(1)}),\dots,Q_{\bs{N}}(f_n^{(m)}) )]\big| \leq C \sqrt{\max_{j=1,\dots,m}\tau_n(f_n^{(j)})}. \\$$ 
\end{thm}

\begin{rmk}
We stress that, due to our normalization assumption on the kernels $f_n^{(j)}$, $\sqrt{\max\limits_{j=1,\dots,m}\tau_n(f_n^{(j)})} \leq 1$.\end{rmk}

The main result of the present subsection is a multidimensional version of Theorem \ref{superTeo1}, stated via Theorem \ref{ComponentJoint}: the proof we will provide 
use the findings of  \cite[Proposition 2]{PeccatiTudor} (stated in Theorem \ref{TeoPeccatiTudor}), where it is shown that for vectors of the type $(Q_{\bs{N}}(f_n^{(1)}),\dots,Q_{\bs{N}}(f_n^{(m)}) )$, joint convergence towards the multidimensional normal distribution is equivalent to componentwise central convergence, as summarized in the next statement. Note that the original statement does not concern exclusively Gaussian homogeneous sums, but deals with vectors of multiple Wiener integrals of symmetric functions in full generality. 

\begin{thm}\label{TeoPeccatiTudor}
For $d\geq 2$ and $m\geq 1$, assume that $C=(C_{i,j})_{i,j=1,\dots,m}$ is a real valued, positive definite, symmetric matrix. For every $j=1,\dots,m$, let $Q_{\bs{X}}(f_n^{(j)})$ be a sequence of homogeneous sums of degree $d$, with $f_n^{(j)}:[n]^d \rightarrow \mathbb{R}$ symmetric kernel, vanishing on diagonals, such that:
$$ \lim_{n \rightarrow \infty}\E[ Q_{\bs{X}}(f_n^{(j)}) Q_{\bs{X}}(f_n^{(i)})] = C_{i,j}\quad \forall \, i,j=1,\dots,m.$$
Then, the following statements are equivalent as $n\rightarrow \infty$:
\begin{itemize}
\item[(i)] $Q_{\bs{N}}(f_n^{(j)}) \stackrel{\text{ Law }}{ \longrightarrow} \mathcal{N}(0,C_{j,j})$ for every $j=1,\dots,m$;
\item[(ii)] $(Q_{\bs{N}}(f_n^{(1)}),\dots,Q_{\bs{N}}(f_n^{(m)}) ) \stackrel{\text{ Law }}{ \longrightarrow} \mathcal{N}(0, C)$, with $\mathcal{N}(0,C)$ denoting the $m$-dimensional Gaussian distribution with covariance matrix given by $C$.\\
\end{itemize}
\end{thm}

Combining  Theorem \ref{superTeo1}, Theorem \ref{MultiMossel} and Theorem \ref{TeoPeccatiTudor}, it is possible to conclude that  the equivalence between joint and componentwise convergence for normal approximations of random vectors $(Q_{\bs{X}}(f_n^{(1)}),\dots,Q_{\bs{X}}(f_n^{(m)}) )$  always holds true under the assumptions $\E[X^3]=0$ and $\E[X^4] \geq 3$, as made precise in the following statement.

\begin{thm}\label{ComponentJoint}
Fix $m\geq 1$ and $d\geq 2$. Let $\bs{X} = \{X_{i}\}_{i\geq 1}$ be a sequence of independent copies of a random variable $X$ verifying Assumption {\bf (1)} and $\E[X^4]\geq 3$. For every $j=1,\dots,m$, let $Q_{\bs{X}}(f_n^{(j)})$ be a sequence of homogeneous sums of degree $d$, with $f_n^{(j)}:[n]^d \rightarrow \mathbb{R}$ admissible kernel, such that:
$$ \lim_{n \rightarrow \infty}\E[ Q_{\bs{X}}(f_n^{(j)}) Q_{\bs{X}}(f_n^{(i)})] = C_{i,j} \quad \forall\, i,j=1,\dots,m \,,$$
where $C= (C_{i,j})_{i,j=1,\dots,m}$ is a real valued, positive definite symmetric matrix. The following statements are equivalent as $n\rightarrow \infty$:
\begin{itemize}
\item[(i)] $Q_{\bs{X}}(f_n^{(j)}) \stackrel{\text{ Law }}{ \longrightarrow} \mathcal{N}(0,C_{j,j})$ for every $j=1,\dots,m$;
\item[(ii)] $(Q_{\bs{X}}(f_n^{(1)}),\dots,Q_{\bs{X}}(f_n^{(m)}) ) \stackrel{\text{ Law }}{ \longrightarrow} \mathcal{N}(0, C)$,
with $\mathcal{N}(0,C)$ denoting the $m$-dimensional Gaussian distribution with covariance matrix given by $C$.
\end{itemize}
\end{thm}

\begin{proof}
It is sufficient to prove  that $(i) \Rightarrow (ii)$, since the reverse implication is obvious.\\
Assume that $(i)$ occurs. Under the assumption $E[X^4] \geq 3$ and by virtue of Theorem \ref{superTeo1}, $X$ satisfies the Fourth Moment Theorem and its law is universal at the order $d$ for normal approximations of homogeneous sums of degree $d$, implying, in particular, that:
$$Q_{\bs{N}}(f_n^{(j)}) \stackrel{\text{ Law }}{ \longrightarrow} \mathcal{N}(0,C_{j,j}) \quad \text{for every } j=1,\dots,m ,$$
for a sequence $\bs{N}$ of independent standard Gaussian random variables.  Besides, for every $j=1,\dots,m$, $\tau_n^{(j)} = \max\limits_{i=1,\dots,n}\mathrm{Inf}_i(f_n^{(j)}) \longrightarrow 0$ as $n \rightarrow \infty$. Since 
$$\E[Q_{\bs{X}}(f_n^{(j)}) Q_{\bs{X}}(f_n^{(i)})] = \E[Q_{\bs{N}}(f_n^{(j)}) Q_{\bs{N}}(f_n^{(i)})] \; \forall i,j=1,\dots,m,$$
by virtue of Theorem \ref{MultiMossel} the random vectors $(Q_{\bs{N}}(f_n^{(1)}),\dots,Q_{\bs{N}}(f_n^{(m)}) )$ and \linebreak $(Q_{\bs{X}}(f_n^{(1)}),\dots,Q_{\bs{X}}(f_n^{(m)}) )$ are asymptotically close in distribution. Finally, the conclusion follows by applying Theorem \ref{TeoPeccatiTudor}.
\end{proof}

\subsection{Multidimensional CLT in the free setting}

The subsequent statements summarize \cite[Theorem 1.3]{NouSpeiPec} and  \cite[Theorem 1.3]{NourdinDeya}, where the free counterpart to  \cite[Proposition 2]{PeccatiTudor} and to Theorem \ref{invMossel} was achieved  respectively. Even if the original statement deals with vectors of Wigner stochastic integrals in full generality, the statement here is recalled only for semicircular homogeneous sums.

\begin{thm}\label{Equi_Free}
For $d \geq 2$ and $m\geq 1$, let $f_n^{(j)}:[n]^d \rightarrow \mathbb{R}$ be a mirror symmetric function for every $j=1,\dots,m$. If $\mathbf{S} = \{S_i\}_{i \geq 1}$ denotes a sequence of freely independent standard semicircular random variables, let $C = (C_{i,j})_{i,j=1,\dots,m}$ be a real-valued, positive definite symmetric matrix, such that for $i,j=1,\dots,m$,
$$ \lim_{n\rightarrow \infty} \varphi\big(Q_{\mathbf{S}}(f_n^{(i)}) Q_{\mathbf{S}}(f_n^{(j)}) \big) = C_{i,j}.$$
If $(s_1,\dots,s_m)$ denotes a semicircular system with covariance determined by $C$, the following statements are equivalent as $n\rightarrow \infty$:
\begin{itemize}
\item[(i)] $Q_{\mathbf{S}}(f_n^{(j)}) \stackrel{ \text{ Law }}{\longrightarrow } s_{j}$;
\item[(ii)] $(Q_{\mathbf{S}}(f_n^{(1)}),\dots,Q_{\mathbf{S}}(f_n^{(m)})) \stackrel{ \text{ Law }}{\longrightarrow } (s_1,\dots,s_m).$\\
\end{itemize}
\end{thm}


For $d\geq 2$, the combination between Theorem \ref{Multiinvariance2}, Theorem \ref{superTeo2} and  Theorem \ref{Equi_Free} allows us to prove that Theorem \ref{Equi_Free} itself can be extended to all random variables with non-negative free kurtosis, providing therefore the free counterpart to Theorem \ref{ComponentJoint}. 

\begin{thm}\label{ComponentJointFree}
Fix $m\geq 1$ and $d\geq 2$. Let $\bs{Y} = \{Y_{i}\}_{i\geq 1}$ be a sequence of freely independent copies of a random variable $Y$ verifying Assumption {\bf (1)} and such that $\kappa_4(Y)= \varphi(Y^4)- 2 \geq 0$. For every $j=1,\dots,m$, let $Q_{\bs{Y}}(f_n^{(j)})$ be a sequence of homogeneous sums of degree $d$, with $f_n^{(j)}:[n]^d \rightarrow \mathbb{R}$ symmetric, vanishing on diagonals kernels such that: 
$$ \lim_{n \rightarrow \infty}\varphi\big( Q_{\bs{Y}}(f_n^{(j)}) Q_{\bs{Y}}(f_n^{(i)})\big) = C_{i,j} \quad \forall i, j=1,\dots,m. $$
If $C=(C_{i,j})_{i,j=1,\dots,m}$ is a real-valued, positive definite symmetric matrix, and $(s_1,\dots,s_m)$ denotes a semicircular system with covariance determined by $C$, the following statements are equivalent as $n\rightarrow \infty$:
\begin{itemize}
\item[(i)] $Q_{\bs{Y}}(f_n^{(j)}) \stackrel{\text{ Law }}{ \longrightarrow} s_j$ for every $j=1,\dots,m$;
\item[(ii)] $(Q_{\bs{Y}}(f_n^{(1)}),\dots,Q_{\bs{Y}}(f_n^{(m)}) ) \stackrel{\text{ Law }}{ \longrightarrow} (s_1,\dots,s_m)$.
\end{itemize}
\end{thm}

\begin{proof}
It is sufficient to prove  that $(i)$ $\Rightarrow$ $(ii)$, since the reverse implication is obvious.\\
Assume that $(i)$ occurs. Under the assumption $\varphi(Y^4) \geq 2$ and by virtue of Theorem \ref{superTeo2}, $Y$ satisfies the Fourth Moment Theorem and its law is universal for semicircular approximations of homogeneous sums, at the given order $d$. In particular one has $Q_{\SSw}(f_n^{(j)}) \stackrel{\text{ Law }}{ \longrightarrow} s_j$ for every $j=1,\dots,m$; besides, $\tau_n^{(j)} = \max\limits_{i=1,\dots,n}\mathrm{Inf}_i(f_n^{(j)}) \longrightarrow 0$ for every $j=1,\dots,m$. Finally, since 
$$\varphi\big(Q_{\bs{Y}}(f_n^{(j)}) Q_{\bs{Y}}(f_n^{(i)})\big) = \varphi\big(Q_{\bs{S}}(f_n^{(j)}) Q_{\bs{S}}(f_n^{(i)})\big) \qquad \forall i, j =1,\dots,m\,,$$
by virtue of Theorem \ref{Multiinvariance2} it follows that the vectors $(Q_{\bs{S}}(f_n^{(1)}),\dots,Q_{\bs{S}}(f_n^{(m)}) )$ and \linebreak $(Q_{\bs{Y}}(f_n^{(1)}),\dots,Q_{\bs{Y}}(f_n^{(m)}) )$ are asymptotically close in distribution: hence the conclusion follows by Theorem \ref{Equi_Free}.\\
\end{proof}

The last statement that we need to recall is \cite[Theorem 1.6]{NouSpeiPec}, where the authors established the following \textit{transfer principle} for the multidimensional CLT between Wiener and Wigner chaos, here formulated only for homogeneous sums.

\begin{thm}\label{PrimoTransfer}
Let $d \geq 1$ and $m\geq 1$ be fixed integers, and let $C = (C_{i,j})_{i,j=1,\dots,m}$ be a real-valued, positive definite symmetric matrix. For every $j=1,\dots, m$, let $f_n^{(j)}:[n]^d \rightarrow \mathbb{R}$ be an admissible kernel, and assume that, for every $i,j=1,\dots,m$:
$$ d!\varphi(  Q_{\bs{S}}(f_n^{(i)})  Q_{\bs{S}}(f_n^{(j)})  ) \rightarrow C_{i,j},$$
$$ \E[  Q_{\bs{N}}(f_n^{(i)})  Q_{\bs{N}}(f_n^{(j)})  ]\rightarrow C_{i,j},$$
where $\bs{S}$ denotes a sequence of freely independent standard semicircular random variables, and $\bs{N}$ denotes a sequence of independent standard Gaussian random variables. Then, if $(s_1,\dots,s_m)$ denotes a semicircular system, with covariance given by $C$, and $\mathcal{N}(0,C)$ a denotes the multivariate normal distribution of covariance $C$, the following statements are equivalent as $n\rightarrow \infty$:
\begin{itemize}
\item[(i)] $(\sqrt{d!}Q_{\bs{S}}(f_n^{(1)}),\dots, \sqrt{d!} Q_{\bs{S}}(f_n^{(m)}) ) \stackrel{\text{Law}}{\rightarrow} (s_1,\dots,s_m)$ 
\item[(ii)] $(Q_{\bs{N}}(f_n^{(1)}), \dots,  Q_{\bs{N}}(f_n^{(m)}) )  \stackrel{\text{Law}}{\rightarrow}  \mathcal{N}(0,C)$.
\end{itemize}
\end{thm}

Thanks to Theorems \ref{ComponentJoint} and \ref{ComponentJointFree}, Theorem \ref{PrimoTransfer} can be completely generalized to a transfer principle for central convergence, between homogeneous sums $Q_{\X}(f_n)$, with $X$ satisfying Assumption {\bf (1)} and with non-negative kurtosis, over a classical probability space, and free homogeneous sums $\sqrt{d!}Q_{\Y}(f_n)$, with $Y$ satisfying Assumption {\bf (2)} and with non-negative free kurtosis, over a free probability space $(\mathcal{A},\varphi)$. Note that, since we are analysing the occurrence of the fourth moment phenomenon along with the universality property, we need to set $d\geq 2$.

\begin{thm}\label{Transfer}
Set $d\geq 2$. Let $X$ be a random variable (in the classical sense), satisfying Assumption {\bf (1)} and such that $\E[X^4]\geq 3$, and $Y$ be a free random variable  satisfying Assumption {\bf (2)} and $\varphi(Y^4)\geq 2$. Let $m\geq 1$, and for every $j=1,\dots,m$, let $f_n^{(j)}:[n]^d \rightarrow \mathbb{R}$  be a symmetric, vanishing on diagonal kernel, such that:
$$ \lim_{n\rightarrow\infty}d!\varphi\big(Q_{\Y}(f_n^{(i)})Q_{\Y}(f_n^{(j)})\big) = \dfrac{1}{d!} \lim_{n\rightarrow\infty}\E[Q_{\X}(f_n^{(i)})Q_{\X}(f_n^{(j)})] = C_{i,j}, \; \forall i,j=1,\dots,m,$$
with $C=(C_{i,j})_{i,j=1,\dots,m}$  real-valued, positive definite symmetric matrix. Then the following conditions are equivalent as $n \rightarrow \infty$:
\begin{itemize}
\item[(i)] $\big( Q_{\X}(f_n^{(1)}),\dots,  Q_{\X}(f_n^{(m)}) \big)  \stackrel{\text{Law}}{\longrightarrow} \mathcal{N}(0,C)$;
\item[(ii)] $\big( \sqrt{d!}Q_{\Y}(f_n^{(1)}),\dots,  \sqrt{d!}Q_{\Y}(f_n^{(m)})\big) \stackrel{\text{Law}}{\longrightarrow} (s_1,\dots,s_m)$,
\end{itemize}
with $(s_1,\dots,s_m)$ denoting a semicircular system with covariance determined by $C$.
\end{thm}

\begin{proof}
Assume first that $(i)$ holds: then, for every $j=1,\dots, m$, $Q_{\X}(f_n^{(j)}) \stackrel{\text{Law}}{\longrightarrow} \mathcal{N}(0,C_{j,j})$, implying,  by virtue of Theorem \ref{superTeo1}, that $Q_{\NN}(f_n^{(j)}) \stackrel{\text{Law}}{\longrightarrow} \mathcal{N}(0,C_{j,j})$. By virtue of Theorem \ref{TeoPeccatiTudor}, then, we have the joint convergence 
$(Q_{\NN}(f_n^{(1)}),\dots,Q_{\NN}(f_n^{(m)})) \stackrel{\text{Law}}{\longrightarrow} \mathcal{N}(0,C)$, which is, in turn, equivalent to the joint convergence $\big( \sqrt{d!}Q_{\SSw}(f_n^{(1)}),\dots,  \sqrt{d!}Q_{\SSw}(f_n^{(m)})\big) \stackrel{\text{Law}}{\longrightarrow} (s_1,\dots,s_m)$, by virtue of \cite[Theorem 1.6]{NouSpeiPec}. Finally, Theorem \ref{Equi_Free}  implies that \linebreak $\sqrt{d!}Q_{\Y}(f_n^{(j)}) \stackrel{\text{Law}}{\longrightarrow} s_j$ and the conclusion follows by Theorem \ref{ComponentJointFree}.

To prove the reverse implication, start with Theorem \ref{superTeo2} and consider Theorem \ref{ComponentJoint} instead of Theorems \ref{superTeo1} and \ref{ComponentJointFree}, respectively.\\
\end{proof}

\begin{rmk}
Due to the assumption $\E[X^3]=0$ in the statement of Theorem \ref{superTeo1}, this setting does not fit the Poisson homogeneous chaos. In view of the Transfer principle provided with Theorem \ref{Transfer}, this failure is consistent with the lack of a transfer principle,  for central convergence, between classical and free Poisson chaos, as highlighted with a counterexample in \cite{Solesne}.
\end{rmk}

\bibliographystyle{plain}

\end{document}